\documentclass[11pt]{amsart}

\usepackage[arrow,matrix,curve,cmtip,ps]{xy}
\usepackage{amsmath,amssymb,amsfonts,amsthm,amscd}
\usepackage{enumerate}

\newtheorem{theorem}{Theorem}[section]
\newtheorem{proposition}[theorem]{Proposition}
\newtheorem{lemma}[theorem]{Lemma}
\newtheorem{corollary}[theorem]{Corollary} 
\theoremstyle{remark}
 
\theoremstyle{definition}
\newtheorem{definition}[theorem]{Definition}

\numberwithin{equation}{section} 
\newcommand{\ip}[2]{\left\langle #1,#2  \right\rangle}

\newcommand{\norm}[1]{\left \| #1\right\|}
\newcommand{\mnorm}[1]{\left \| \left [ #1\right ]\right \|}
\newcommand{\abs}[1]{\left \lvert #1 \right \lvert}

\newcommand{\Min}[1]{{\rm Min}(#1)}
\newcommand{\Max}[1]{{\rm Max}(#1)}
\newcommand{\Ran}[1]{{\rm Ran}(#1)}
\newcommand{\Ker}[1]{{\rm Ker}(#1)}
\newcommand{\Ball}[1]{{\rm Ball}(#1)}

\newtheorem*{xrem}{Remark}

\DeclareMathOperator{\iotimes}{\check{\otimes}}

\DeclareMathOperator{\Rdb}{{\mathbb R}}

\DeclareMathOperator{\Al}{{\mathcal{A}_{\ell}}}

\DeclareMathOperator{\Ml}{{\mathcal{M}_{\ell}}}

\begin{document}

\title{Real Operator Algebras and Real Completely Isometric Theory} 
\date{\today} 
\author{Sonia Sharma}
\address{Department of Mathematics, SUNY Cortland, Cortland, NY
  13045} \email[Sonia Sharma]{sonia.sharma@cortland.edu}
\begin{abstract}
This paper is a continuation of the program started by Ruan in \cite{RuaR} and \cite{RuaC}, of developing real operator space theory. In particular, we develop the theory of real operator algebras. 
 We also show among other things that the injective envelope, $C^*$-envelope and non-commutative Shilov boundary exist for a real 
operator space. We develop real one-sided $M$-ideal theory and characterize  one-sided $M$-ideals in real $C^*$-algebras and real operator algebras with contractive approximate identity. 
\end{abstract}

\maketitle
\section{Introduction}
In functional analysis, the underlying objects of study are vector spaces over a field, where the field is usually either the field of real numbers, $\mathbb{R}$, or the field of complex numbers, $\mathbb{C}$. The field of complex numbers has been preferred over real numbers, since the field of real numbers is a little more restrictive. 
For instance, every polynomial over the field of reals
has a roots in $\mathbb{C}$, but need not have any root in $\mathbb{R}$, or a $n\times n$ matrix need not have real eigenvalues. Thus, usually most of the theory is developed with the assumption that the underlying field is $\mathbb{C}$. The theory of real spaces, however, occurs naturally in all areas of mathematics and
physics. They come up naturally in the theory of $C^*$-algebras, for
instance the self-adjoint part of every $C^*$-algebra is a real
space, also in graded $C^*$-algebras and in the theory of real TROs in graded
$C^*$-algebras. See {\cite{Goo, Li, Sch, Sto}} for the theory of real $C^*$-algebras and real $W^*$-algebras. 
They also come up in $JB^*$-triples {\cite{Russo}} and 
KK theory {\cite{Bla, Guent}}. Thus it becomes important to study the analogues theory for the case
when the field is the real scalars and know which results hold true and which results fail.

The theory of real operator spaces is the study of subspaces of bounded operators on
real Hilbert spaces. In the general theory of (complex) operator spaces, the underlying Hilbert space
is assumed to be a complex Hilbert space. 

In two recent papers {\cite{RuaR, RuaC}}, Ruan studies the basic theory of real operator spaces. 
He shows that with appropriate modifications, many
complex results hold for real operator spaces. It is shown among
other things, that
Ruan's characterization, Stinespring's theorem, Arveson's extension
theorem, and injectivity of $B(H)$ for real Hilbert space $H$,
 hold true for real operator spaces.  
 In \cite{RuaC}, Ruan defines the notion of complexification of a real operator space
 and studies the relationship between the properties of real operator spaces and their complexification. 
We want to continue this program, and develop more theory of real operator spaces and real
 operator algebras since
there appears to be a gap in the literature here. This is unfortunate because 
a researcher who is familiar with the complex operator space theory and is facing a problem which involves real operator spaces, must then face the daunting task of reconstructing a large amount of the real theory from scratch.

In section 2, we further develop the real operator space theory, investigate operator space structures such as minimal and maximal operator spaces structures on real operator spaces. We also see that unlike in the Banach space theory (see \cite[Proposition 1.1.6]{Li}), if $X$ is a complex operator space then $(X^*)_r$ is not completely isometrically isomorphic to $(X_r)^*$, where $X_r$ denotes $X$ considered as a real operator space.
Also, unlike in the complex operator space theory, $\ell_2(\mathbb{R})$ does not have a unique operator space structure.
In section 3, we briefly consider real operator algebras and their
 complexification.  We show that the BRS characterization theorem and Meyer's Theorem hold for real operator algebras. 
In section 4,  we  study the relation between the real injective envelope and 
the injective envelope of its complexification. We also show among other things that the injective envelope, $C^*$-envelope and non-commutative Shilov boundary exist for a real 
operator space. In Section 5, we begin to develop the theory of real one-sided $M$-ideals and show that several results from one-sided $M$-ideal theory \cite{BEZ}, are true in the real case. We also show that 
a subspace $J$ in a real operator space $X$ is a right $M$-ideal if and only if $J_c$ is a (complex) right $M$-ideal in $X_c$, which allows us to characterize one-sided $M$-ideals in real $C^*$-algebras and in real operator algebras with one-sided contractive approximate identity. We also infer that a real operator space $X$ is $M$-embedded if and only if $X_c$ is $M$-embedded. This facilitates in generalizing results in one-sided $M$-embedded theory from \cite{Sha2} to real operator spaces. For instance, we show that real one-sided $M$-embedded TRO are of the form $A\cong \oplus_{i,j} ^{\circ}\mathbb{K}(H_i, H_j)$ completely isometrically, for some real Hilbert spaces $H_i$, $H_j$. 

Much of the work presented in this paper was done in author's thesis in 2009 (see \cite{Sha1}).

\section{Real Operator Spaces}
A (concrete) {\em real operator space} is a closed subspace of $B(H)$, for some real
Hilbert space $H$. An abstract real operator space is a pair 
$(X,\norm{.}_n)$, where $X$ is a real vector space such that
there is a complete isometry $u: X\longrightarrow B(H)$, for some real
Hilbert space $H$. As in the case of complex operator spaces,
Ruan's norm characterization hold for real operator spaces, and we
say that $(X,\norm{.}_n)$ is an abstract real
operator space if and only if it satisfies 
\begin{enumerate}[\upshape (i)]
\item $\norm{x\oplus y}_{n+m}=\max \{ \norm{x}_n, \norm{y}_m\}$,
\item $\norm{\alpha x\beta}_n \leq \norm{\alpha}\norm{x}_n \norm{\beta}$,
\end{enumerate}
for all $x\in M_n(X)$, $y\in M_m(X)$ and $\alpha, \beta \in M_n(\mathbb{R})$.

 Let $X\subset B(H)$,
then $X_c\subset B(H)_c$ and $B(H)_c\cong B(H_c)$ completely isometrically, where $H_c$ is a complex Hilbert space
(see e.g.\ discussion on page 1051 from \cite{RuaC}).
Thus there is a canonical matrix norm structure on $X_c$ inherited from $B(H_c)$, and
$X_c$ is a
complex operator space with this canonical norm structure. The space $B(H_c)$ can
be identified with a real subspace of $M_2(B(H))$ via 
\begin{equation}\label{eq2}
B(H_c)=B(H)+iB(H)=
                 \left\{\left[\begin{array}{cc} x & -y \\ 
                                               y & x \end{array}\right]\ :\ x,y\in B(H)\right\} \in M_2(B(H)).
\end{equation} 
Thus the matrix norm on the complexification
is given by
$$\mnorm{x_{kl}+iy_{kl}}=
                     \mnorm{\begin{array}{cc} x_{kl} & -y_{kl} \\ y_{kl} & x_{kl} \end{array}},$$
and we have the following complete isometric identification
$$X_c=\left\{\left[\begin{array}{cc} x & -y \\ 
                                               y & x \end{array}\right]\ :\ x,y\in X \right \} \in M_2(X).$$ 

This canonical complex operator space matrix norm structure on its complexification
$X_c=X+iX$ which extends the original norm on $X$, i.e., $\norm{x+i0}_n=\norm{x}_n$ and satisfies the {\em reasonable} (in the sense of {\cite{RuaC}}) condition
$$\norm{x+iy}=\norm{x-iy},$$
for all $x+iy\in M_n(X_c)=M_n(X)+iM_n(X)$ and $n\in \mathbb{N}$. 
By {\cite[Theorem 2.1]{RuaC}}, the operator space structure on $X_c$ is independent of the choice of $H$.  
Moreover, by {\cite[Theorem 3.1]{RuaC}}, any other reasonable {in the above sense} operator space structure on $X_c$ is completely isometric to the canonical operator space structure on $X_c$. 

Let $T$ be a (real linear) bounded operator between real operator spaces $X$ and $Y$, then 
define the complexification of $T$, $T_c:X_c\longrightarrow Y_c$
as $T_c(x+iy)=T(x)+iT(y)$, a complex linear bounded operator. It is shown in {\cite[Theorem 3.1]{RuaC}} that
if $T$ is a complete contraction (respectively, a complete isometry) then $T_c$ is a complete contraction (respectively, complete isometry) with $\norm{T_c}_{cb}=\norm{T}_{cb}$. This is
 not true in the case of a Banach space, that is, the complexification
 of a contraction on a real Banach space is bounded, but is not
 necessarily a contraction, and $\norm{T} \neq \norm{T_c}$, in general.
If $\pi:X_c\longrightarrow Y_c$ is linear, then as in \cite{RuaC}, define a linear map $\overline{\pi}:X_c\longrightarrow Y_c$ 
as $\overline{\pi}(x+iy)=\overline{\pi(x-iy)}$. Let ${\rm Re}(\pi)=\frac{\pi+\overline{\pi}}{2}$ and
let ${\rm Im}(\pi)=\frac{\pi-\overline{\pi}}{2i}$. Then ${\rm Re}(\pi)$ and ${\rm Im}(\pi)$ are (real) linear maps which map $X$
into $Y$ such that $\overline{{\rm Re}(\pi)}={\rm Re}(\pi)$, $\overline{{\rm Im}(\pi)}={\rm Im}(\pi)$, and $\pi={\rm Re}(\pi)+i{\rm Im}(\pi)$.

\subsection{Minimal Real Operator Space Structure}

A real $C^*$-algebra is a norm closed $*$-subalgebra of $B(H)$, where $H$ is a real Hilbert space. By \cite[Proposition 5.13]{Li},
every real $C^*$-algebra $A$ is a fixed point algebra of $(B,-)$, i.e., 
$A=\{b\in B \ :\ \overline{b}=b\}$, where $B$ is a (complex) $C^*$-algebra, and
``-'' is a conjugate linear $*$-algebraic isomorphism of $B$ with period $2$. Moreover,
$B=A+iA$ is the complexification of $A$. 

Let $A$ be a commutative real $C^*$-algebra. Then define the {\em spectral space} of $A$ as,
 $$\Omega=\{\rho|_A \ :\ {\rm \rho \ is \ nonzero\ multiplicative\ linear\ functional\ on\ A_c}\}.$$
 In other words, $\Omega$ is the set of all non-zero complex valued multiplicative real linear functionals
on $A$. Then using the ``$-$'' on $A_c$, define ``$-$'' on $\Omega$ as,
$$\overline{\rho} (a)=\overline{\rho (a)}.$$
 
Then every commutative real $C^*$-algebra $A$ is of the form
$$A\cong C_0(\Omega,-)=\{f\in C_0(\Omega):f(\overline{t})=\overline{f(t)}\; \forall \; t\in \Omega \},$$ 
where $\Omega$ is the spectral space of $A$,
and ``$-$'' is a conjugation on $\Omega$ defined above (see e.g. \cite[5.1.4]{Li}).
Also $A_c=C_0(\Omega)$. 

If $\Omega$ is any compact Hausdorff space then there is a canonical real $C^*$-algebra, 
$$C(\Omega, \mathbb{R})=\{f: \Omega \longrightarrow \mathbb{R}: {\rm f\ is\ continuous}\}.$$ 
 For instance, if $A$ is a commutative real $C^*$-algebra and $a\in A$ is self adjoint, then the real $C^*$-algebra generated by $a$ in $A$, $C^*(a)$, is of the form $C(\Omega, \mathbb{R})$, where $\Omega$ is the spectral space of $A$. 
 But not every
commutative real $C^*$-algebra $A$ is of the form $C(\Omega, \mathbb{R})$. To see this, let $\Omega=S^2\subset \mathbb{R}^3$, the 
3-dimensional sphere. Let $A=\{f:\Omega\longrightarrow \mathbb{C}: \ f(-t)=\overline{f(t)}\ \forall \ t\in \Omega\}$. Then $A_c=C(\Omega)$, so $A$ is a real $C^*$-algebra. But $A$ is not $*$-isomorphic to $C(\Omega,\mathbb{R})$
since $A_{\rm sa}=
\{f: \Omega \longrightarrow \mathbb{R}: \ f(t)=f(-t)\}\ncong C(\Omega,\mathbb{R})$. 

We can define an operator
space structure on $C(\Omega,-)$ by the canonical structure it inherits as a
subspace of $C(\Omega)$.
 Then $C(\Omega)$ is the operator space
complexification of $C(\Omega,-)$, in the sense defined
above (see e.g., \cite[Proposition 5.1.3]{Li}).
Let $E$ be a real Banach space. Then $E$ can be embedded isometrically
into a real commutative $C^*$-algebra $A$ of the form
$C(\Omega,\mathbb{R})$. For instance, take $\Omega=\Ball{E^*}$ where
$E^*=\{ f: E\longrightarrow \mathbb{R}$, $f$ continuous$\}$.
Since commutative real $C^*$-algebras are real operator spaces, there is 
an operator space matrix structure on $E$ via the identification
$M_n(E)\subseteq M_n(A)$. 
This operator space structure is called the minimal operator space structure since it 
is the smallest 
operator space structure on $E$. 
To see this,
let $E$ be the operator space sitting inside
$C(\Omega,\mathbb{R})$ and $F$ denote the Banach space $E$, with a different operator
space structure.  
Let $u:F\longrightarrow E$ be the identity map. So
$u$ is an isometry, $\norm{u(x)}_E=\norm{x}_E=\norm{x}_F$, and for any $[x_{ij}]\in M_n(E)$ and $\Omega=\Ball{E^*}$, 
\begin{eqnarray*}
  \norm{u_{n}[x_{ij}]}_{M_n(E)}&=&\norm{[u(x_{ij})]}_{M_{n(E)}}\\
  &=&\sup \{ \norm{[u(x_{ij})(t)]}_{M_n(\mathbb{R})}:t\in \Omega \}\\
  &=&\sup \{\lvert{\sum_{{i,j}}{u(x_{ij})(t)w_jv_i}}\rvert 
                            :\vec{v},\vec{w}\in l^2_n(\mathbb{R}),\; t\in \Omega \}\\
  &=&\sup\{ \| u(\sum_{i,j}{x_{ij}w_jv_i})\|_{E} :\vec{v},\vec{w}\in l^2_n(\mathbb{R}) \}\\
  &=&\sup\{{\| \sum_{i,j}{x_{ij}w_jv_i}}\|_{F} :\vec{v},\vec{w}\in l^2_n(\mathbb{R}) \} \\
  &\leq & \|{{[x_{ij}]}}\| _{M_{n}(F)}.                         
\end{eqnarray*}
This implies that $\norm{u:F\longrightarrow E}_{cb} \leq 1$.

Let $E$ be a Banach space and let $x, y\in E$. Define
\begin{eqnarray*}
      \norm{x+iy} 
           =\sup\{\| \alpha x+\beta y \| : 
                    \;\alpha^{2}+\beta^{2}\leq 1,\; \alpha,\beta \in \mathbb{R}\}.
\end{eqnarray*}
Then $\norm{x+iy}=\norm{x-iy}$ and
    $\norm{x+i0}=\norm{x}$, and thus with this new norm $E_c$ is a
    complexification of the Banach space $E$. This norm is called the $w_2$-norm in \cite{DF}.
Also note that for any $z+iw\in \mathbb{C}$, 
$$\abs{z+iw} 
           =\sup\{\abs{\alpha z+\beta w} : 
                    \;\alpha^{2}+\beta^{2}\leq 1,\; \alpha,\beta \in \mathbb{R}\}.$$
So, \begin{eqnarray*}
            \norm{x+iy}&=& \sup\{ \abs{\alpha f(x)+\beta f(y)} : 
                    \;\alpha^{2}+\beta^{2}\leq 1,\; \alpha,\beta \in \mathbb{R} \ {\rm and} \ f\in {\rm Ball}(E^*) \}\\
            &=& \sup \{ \abs{f(x)+if(y)} \ : \ f\in {\rm Ball}(E^*) \} \\   
            &=&\sup \left \{ \left \| \left [ 
                                \begin{array}{cc} 
                                       f(x) & -f(y)\\ 
                                       f(y) & f(x)
                                 \end{array}
                            \right ] \right \|  : f\in {\rm Ball}(E^*) \right \}.
\end{eqnarray*}                    

\begin{proposition} 
Let $E$ be a real Banach space and $E_c$ be the complexification of $E$ with the norm defined above. Then $(\Min{E})_c= \Min{E_c}$, completely isometrically.
\end{proposition}
\begin{proof}
Let $\pi: E\longrightarrow C(\Omega, \mathbb{R})$ be the canonical isometry. Then 
$\Min{\pi}: \Min{E}\longrightarrow C(\Omega, \mathbb{R})$ is a complete isometry, and so, 
$\Min{\pi}_c: \Min{E}_c\longrightarrow  C(\Omega)$ is a complete isometry. Further,
\begin{eqnarray*}
   \norm{\pi_c(x+iy)}
        &=&\sup\{\lvert{\pi(x)(f)+i\pi(y)(f)}\rvert:f\in \Omega=\Ball{E^*}\}\\
        &=&\sup\{\lvert{f(x)+if(y)}\rvert:f\in \Ball{E^*}\}\\
        &=&\sup\{\norm{\alpha x+\beta y}: 
           \ \alpha^{2}+\beta^{2}\leq 1, \ \alpha,\beta \in \mathbb{R}\}\\
        &=& \norm{x+iy}_{E_c}.
\end{eqnarray*}
So $\pi_c: E_c\longrightarrow C(\Omega)$ is an isometry, and hence $\Min{\pi_c}: \Min{E_c}\longrightarrow C(\Omega)$ 
is a complete isometry. So we have the following diagram which commutes.
$$\xymatrix {\Min{E}_c \ar[r]^{\rm c.i.} \ar[d]_{\rm Id} & C(\Omega) \ar[d]^{\rm Id} \\
  \Min{E_c} \ar[r]_{{\rm c.i.}} & C(\Omega) } $$
Hence $(\Min{E})_c= \Min{E_c}$, completely isometrically. 
\end{proof}

\begin{lemma}\label{lemma1}
 Let $A$ and $B$ be real $C^*$-algebras, and let $\pi: A \longrightarrow B$ be a homomorphism. Then $\pi$ is a $*$-homomorphism if and only if 
it is completely contractive. Further, $\pi$ is a complete isometry if and only if it is one-one.
\end{lemma}
\begin{proof}
Let $\pi: A \longrightarrow B$ be a $*$-homomorphism, then $\pi_c: A_c \longrightarrow B_c$ is 
 a $*$-homomorphism. Hence $\pi_c$ is a complete contraction, by \cite[Proposition 1.2.4]{BLM}, so
$\pi=\pi_c|_{A}$ is a complete contraction. A similar argument using the complexification proves the converse. The last assertion follows from \cite[Proposition 1.2.4]{BLM}
and that $\pi_c$ is one-one if $\pi$ is. 
\end{proof}

The following proposition has been noted in \cite{RuaC}.  
\begin{proposition}\label{dualcomp}
Let $X$ be a real operator space, then $(X_c)^*=(X^*)_c$, completely isometrically.
\end{proposition}

\begin{proposition}\label{cdembedd}
 If $X$ is a real operator space then $X\subset X^{**}$ completely isometrically
via the canonical map $i_X$.
\end{proposition}
\begin{proof}
Let $X$ be a real operator space and let $\pi: X_c\hookrightarrow (X_c)^{**}$ be the canonical embedding. By Proposition 
\ref{dualcomp}, $(X_c)^{**}=(X^{**})_c$, completely isometrically via, say, $\theta$. Then 
$\theta \circ \pi$ is a complete isometry such that $(\theta \circ \pi)(z)={\rm Re}(\pi(z))+i{\rm Im}(\pi(z))$, for
all $z\in X_c$. So
the restriction of $\theta \circ \pi$ to $X$ is a complete isometry on $X$ such that, for all $f\in X^*$ and $x\in X$, 
$\overline{(\theta \circ \pi)}(x)=(\theta \circ \pi)(x)$, and
$((\theta \circ \pi)(x))(f)=({\rm Re}(\pi(x)))(f)=f(x)=i_X(x)$. Thus $i_X$ is a complete isometry.
\end{proof}

The maximal operator space structure is the largest operator space
structure that can be put on a real operator space, and its matrix
norms are defined exactly as in the complex case.
$$ \mnorm{x_{{i}{j}}}= 
     \sup \{ \mnorm{u(x_{ij})} : \ u\in \Ball{B(E,Y)},\ Y \; \text{a real operator space} \}.$$

 If we put the maximal operator space structure on $E$, then it has the
universal property that for any real operator space $Y$, and
$u:E\longrightarrow Y$ bounded linear, we have
$$\norm{u:E\longrightarrow Y}=\norm{u: \Max{E}\longrightarrow Y}_{cb}$$
i.e., $B(E,Y)=CB(\Max{E},Y)$

\begin{lemma}\label{rcdual}
Let $K$ be a compact Hausdorff space then $C(K, \mathbb{R})^{**}$ is a (real) commutative 
$C^*$-algebra of the form $C(\Omega,\mathbb{R})$. 
\end{lemma}
\begin{proof}
Let $u:C(K,\Rdb)\longrightarrow C(K,\Rdb)_c=C(K)$ be the inclusion
  map. Then $u^{**}:C(K,\Rdb)^{**}\longrightarrow C(K)^{**}$ is a
  $*$-monomorphism. The second dual of a (real or complex) commutative 
$C^*$-algebra is a commutative $C^*$-algebra. Let $C(K)^{**}\cong C(\Omega)$, $*$-isomorphically.
 Then
  $C(\Omega ,\Rdb)$ sits inside $C(\Omega)$ as a real space, in fact, as
  the real part such that, $C(\Omega,\Rdb)_c=C(\Omega)$. It is enough to show that
  $u^{**}(f)=\overline{(u^{**}(f))}$ for all $f\in C(K,\mathbb{R})^{**}$. We use a
  ${\rm weak}^*$-density argument. First, note that $u^{**}|_{C(K,\Rdb)}=u$ and $u$ is
  selfadjoint, i.e., $u(g)=\overline{u(g)}$ $\forall$ $g\in
  C(K,\Rdb)$. Let $f\in C(K,\Rdb)^{**}$, then there exists a net
  $\{f_{\lambda }\}$ in $C(K,\Rdb)$ converging ${\rm weak}^*$ to $f$. Then
  $u^{**}(f_\lambda )\buildrel {\rm weak}^*\over\longrightarrow
  u^{**}(f)$. This implies that $u^{**}(f_\lambda )(\omega)$ converges
  pointwise to $u^{**}(f)(\omega)$ in $\mathbb{C}$ for all $\omega \in
  \Omega$. Hence 
$\overline{u^{**}(f_\lambda )(\omega)}\longrightarrow \overline{u^{**}(f)(\omega)}$ 
in $\mathbb{C}$
  for all $\omega \in \Omega$. So $\overline{u^{**}(f_\lambda
    )}\buildrel {\rm weak}^*\over \longrightarrow \overline{u^{**}(f)}$. But
  $u^{**}(f_{\lambda })=u(f_{\lambda })=\overline{u^{**}(f_{\lambda
    })}=\overline{u(f_{\lambda })}$. So ${u^{**}(f_\lambda
    )}\buildrel {\rm weak}^*\over \longrightarrow \overline{u^{**}(f)}$. Hence,
  by uniqueness of limit, $u^{**}(f)=\overline{u^{**}(f)}$. This shows
  that the map $u^{**}$ is real, and hence it maps into $C({\Omega,\Rdb})$. 
Let $f\in C(\Omega, \mathbb{R})=(C(K)^{**})_{\rm sa}$. Let $\{f_{\lambda}\}\in C(K)$ be a net
which converges ${\rm weak}^*$ to $f$. Then $\{\overline{f_{\lambda}}\}$ also converges 
${\rm weak}^*$ to $f$, and so does $g_{\lambda}=\frac{f_{\lambda}+\overline{f_{\lambda}}}{2}\in C(K, \mathbb{R})$.
Thus $f\in \overline{\Ran{u}}^{{\rm weak}^*}\subset \Ran{u^{**}}$. Hence $u^{**}$ maps onto 
$C(\Omega, \mathbb{R})$.
\end{proof}

\begin{proposition}{\label{maxmin}}
  Let $E$ be a real Banach space, then 
$$\Min{E^*}=\Max{E}^* \ \ {\rm and }\ \ \Min{E}^*=\Max{E^*}, \ {\rm completely \ isometrically}.$$
\end{proposition}
\begin{proof}
We have that $M_n(\Max{E}^*)\cong
  CB(\Max{E}, M_n(\Rdb))\cong B(E, M_n(\Rdb ))$, isometrically, for each $n$. 
On the other hand,
  $$M_n(\Min{E^*})\cong M_n(\Rdb){\iotimes} E^* \cong B(E, M_n(\Rdb)),$$
  isometrically, where $\hat{\otimes}$ denotes the Banach space injective tensor
  product. Thus $\Min{E^*}=\Max{E}^*$.

  Let $K$ be a compact Hausdorff space, then
 by Lemma \ref{rcdual}, $$\Min{C(K,\Rdb)}^{**}=C(K,\Rdb)^{**}=C(\Omega,\Rdb).$$ On the
  other hand, $\Min{C(K, \Rdb)^{**}}=\Min{C(\Omega,\Rdb)}=C(\Omega,\Rdb)$. Hence
  $\Min{C(K,\mathbb{R})}^{**}$ $=\Min{C(K,\mathbb{R})^{**}}$.

  Let $E$ be a real Banach space, and suppose that $\Min{E}\hookrightarrow C(K,\Rdb)$
  completely isometrically. By taking the duals, we get the following commuting diagram
$$
\xymatrix{
  C(K,\Rdb)^{**}   & \ar[l] \Min{C(K,\Rdb)^{**}} \\
  \Min{E^{**}} \ar[u]& \ar[l] \Min{E}^{**}. \ar[u] }$$ Let $u$ denote the map from $\Min{E}^{**}$ to $\Min{E^{**}}$. Since all the maps except $u$,
in the above diagram are complete isometries and since the diagram commutes, it forces $u$
to be a complete isometry. Hence
$\Min{E}^{**}=\Min{E^{**}}$. Applying the first identity, we proved above, to $E^*$, we get 
$\Min{E^{**}}=\Max{E^*}^*$. Hence,
$\Max{E}^{**}=\Min{E}^{**}$. 
Let
$X=\Max{E^*}$ and $Y=\Min{E}^*$, then since $X^*=Y^*$, this
implies $X^{**}=Y^{**}$ completely isometrically. By the
commuting diagram below
$$\xymatrix {X \ar[r]^{Id} \ar@{^{(}->}[d] & Y \ar@{^{(}->}[d] \\
  X^{**} \ar[r] & Y^{**} } $$ it is clear that $X=Y$,
completely isometrically. 
\end{proof}

We write $\ell^1_2(\mathbb{R})$ for the two-dimensional real Banach space $\mathbb{R}\oplus_1 \mathbb{R}$, and 
$\ell^{\infty}_2(\Rdb)$ for $\mathbb{R}\oplus_{\infty} \mathbb{R}$. Then $\ell^1_2(\mathbb{R})$ is isometrically isomorphic to
$\ell^{\infty}_2(\Rdb)$ via $(x,y)\mapsto (x+y, x-y)$. We also have that $(\ell^{\infty}_2(\Rdb))^*\cong \ell^{1}_2(\Rdb)$ and 
$(\ell^{1}_2(\Rdb))^*\cong \ell^{\infty}_2(\Rdb)$, isometrically. From \cite{Pau} we know that there is a unique operator space structure
on the two-dimensional complex Banach space, $\ell^1_2(\mathbb{C})$. We see next that this is not true in the case of real operator spaces.

\begin{proposition}
The operator space structure on $l^1_2(\mathbb{R})$ is not unique.
\end{proposition}
\begin{proof}
We consider the maximal and the minimal operator space structures on
  $l^2_1(\mathbb{R})$. Using the facts stated above and Proposition \ref{maxmin}, we have that
$$\Max{l_2^1(\Rdb)}\cong \Max{l_2^{\infty}(\mathbb{R})^*}\cong {\Min{l_2^{\infty}(\Rdb)}}^*={l_2^{\infty}(\Rdb)}^*,$$
completely isometrically.
So the maximal operator space matrix norm on $l_2^1(\Rdb)$ is given by
$$\mnorm{(a_{{i}{j}},b_{{i}{j}})}_{\rm max}=
               \sup \{\mnorm{a_{{i}{j}}d_{{k}{l}}+b_{{i}{j}}e_{{k}{l}}} :
               \ [d_{{k}{l}}],[e_{{k}{l}}]\in \Ball{M_m(\mathbb{R})}, \; m\in \mathbb{N} \}.$$ 
On the other hand,
$\Min{l_2^1(\Rdb)}\cong \Min{l_2^{\infty}(\Rdb)}=l_2^{\infty}(\Rdb)$, completely isometrically via the map $(x,y)\mapsto (x+y,x-y)$. So the matrix
norm on $\Min{l_2^1(\Rdb)}$ is 
$$\mnorm{(a_{{ij}},b_{{ij}})}_{\rm min}= 
                        {\rm max}\{\mnorm{(a_{{i}{j}}+b_{{i}{j}})},\mnorm{(a_{{i}{j}}-b_{{i}{j}})}\}.$$
 It is clear that $\mnorm{(a_{{i}{j}},b_{{i}{j}})}_{\rm min} \leq \mnorm{(a_{{i}{j}},b_{{i}{j}})}_{\rm max}.$ 
Let $A=\left[ \begin{array}{cc} 1 & 0 \\ 0 & -1 \end{array} \right]$ and $B=\left[ \begin{array}{cc} 0 & 1 \\ 1 & 0 \end{array} \right]$. Then
\begin{eqnarray*}
\norm{(A,B)}_{M_{2}(\Min{l_2^1)}}&=&\max\{\norm{A+B},\norm{A-B}\}\\
                          &=&\max\left \{\norm{\left[ 
                             \begin{array}{cc} 1 & 1 \\ 1 & -1 
                             \end{array} \right]},\norm{\left[ 
                             \begin{array}{cc} 1 & -1 \\ -1 & -1 
                             \end{array} \right]}\right \}\\
                          &=&\sqrt{2}.
\end{eqnarray*}
Let $[d_{kl}]=A$ and $[e_{kl}]=B$, then 
 \begin{eqnarray*}
\norm{(A,B)}_{M_{2}(\Max{\ell_2^1})} \geq \mnorm{ \begin{array}{cc} {\left [ \begin{array}{cc} 1 & 0 \\ 0 & -1 \end{array} \right ]} & 0 \\ 0 & -{\left [ \begin{array}{cc} 1 & 0 \\ 0 & -1 \end{array} \right]} \end{array} \right ]  + \left [ \begin{array}{cc} 0 & {\left[ \begin{array}{cc} 0 & 1 \\ 1 & 0 \end{array} \right ]} \\ {\left[ \begin{array}{cc} 0 & 1 \\ 1 & 0 \end{array} \right]} & 0 \end{array}}
 \end{eqnarray*}
 On adding and rearranging the rows and columns we see that this norm is the same as
 \begin{eqnarray*}
\mnorm {\begin{array}{cc} \left[ \begin{array}{cc} 1 & 1 \\ 1 & 1 \end{array} \right] & 0 \\ 0 & {\left[ \begin{array}{cc} 1 & -1 \\ -1 & 1 \end{array} \right]}\end{array} }
&=& \max \left \{ \mnorm{\begin{array}{cc} 1 & 1 \\ 1 & 1 \end{array} },\mnorm{ \begin{array}{cc} 1 & -1 \\ -1 & 1 \end{array}}\right \}\\
&=&2.
\end{eqnarray*}
So $\norm{(A,B)}_{M_{2}(\Max{l_2^1})}\geq 2 >
\sqrt{2}=\norm{(A,B)}_{M_{2}(\Min{l_2^1})}$. Hence, there are
two different operator space structures on
$l_2^1(\Rdb)$.
\end{proof}

If X is a complex operator space then, it is also a real operator
space, and hence we can talk about the dual of $X$ both as a real operator
space $X_r^*$, as well as a complex operator space $X^*$, and ask the
question, whether these two spaces are the same real operator
spaces. Then by {\cite{Li}}, $(X^*)_r$ is isometrically isomorphic to $(X_r)^*$. We see next that these spaces need not be completely isometrically isomorphic.
\begin{proposition}
 Let $X=\mathbb{C}$, be a (complex) operator space with the canonical operator space structure. Then $(X_r)^*$ and $(X^*)_r$ are isometrically isomorphic but not necessarily completely isometrically
isomorphic.
\end{proposition}
\begin{proof}
It follows from \cite[Proposition 1.1.6]{Li} that $(\mathbb{C}_r)^*\cong (\mathbb{C}^*)_r$, isometrically.

 Note
  that $\mathbb{C}^*$ is completely isometrically isomorphic to $\mathbb{C}$ via
  the map $\phi_z\longrightarrow z$. 

  Consider
  the canonical map $\theta:\mathbb{C}^*\longrightarrow
  \mathbb{C}_r^*$ given by $\theta(\phi)={\rm Re}(\phi$). By the
  identification $\mathbb{C}\cong\mathbb{C}^*$, we can view the
  above map as $\theta(z)(y)={\rm Re}(y\bar{z})$. If there
  is any complete isometric isomorphism, say $\psi$, then since $\psi$
  is an onto isometry between 2-dimensional real Hilbert spaces, it is
  unitarily equivalent to $\theta$. Any unitary from
  $\mathbb{C}$ to $\mathbb{C}$, is a rotation by an angle $\alpha$. 
  So, $u$ is multiplication by $e^{i\alpha}$, which is a complete isometry
  with the canonical operator space matrix norm structure on $\mathbb{C}$. Then 
  $\theta=u^{-1}\psi u$ is a complete isometry. Thus $\psi$ is a complete isometry if and only
if $\theta$ is a complete isometry. Hence it is enough to show that $\theta$
   is not a complete isometry.

   Consider
  $x=\left[ \begin{array}{cc} 1 & i \\ 0 & 0 \end{array} \right ].$
  Then $\norm{x}=\sqrt{2}$. Since $\theta_2(x)\in
  M_{2}(\mathbb{C}^*)\cong CB(\mathbb{C}, M_{2}(\Rdb)),$ we have that

  $$\norm{\theta_2(x)}=\sup\{\norm{\theta_2(x)([z_{k l}])}\; :[z_{k
    l}=x_{k l}+iy_{k l}]\in M_n(\mathbb{C})\}.$$

  Consider
 \begin{eqnarray*}
\norm{\theta_2(x)[x_{kl}+iy_{kl}]}&=&\mnorm{ \begin{array} {cc} {\rm Re}[x_{k l}+iy_{k l}] & {\rm Re}[ix_{k l}-y_{k l}]\\ 0 & 0 \end{array}}\\
&=&\mnorm{ \begin{array} {cc} [x_{k l}] & [-y_{k l}]\\ 0 & 0 \end{array}}.
\end{eqnarray*} 

 Let $\vec{v}$ be a row vector of length $2n$, whose first $n$ entries are $\alpha_i$ and the last $n$ entries are $\beta_i$. Then the norm of the square of $\vec{v}$ produced by the action of $\left [ \begin{array} {cc} [x_{k l}] & [-y_{k l}]\\ 0 & 0 \end{array}\right ]$ is given by
\begin{eqnarray*}
  \sum_{k=1}^{n}\abs{\sum_{l=1}^{n}(x_{k l}\alpha_{l}-y_{k l}\beta_{l })}^2 
& \leq& \sum_{k=1}^{n} \abs{\sum_{l=1}^{n}(x_{k l}+iy_{k l})(\alpha_l+i\beta_l)}^2\\
& \leq& \mnorm{x_{k l}+iy_{k l}}^2 .
\end{eqnarray*}
So $ \mnorm{ \begin{array} {cc} [x_{k l}] & [-y_{k l}]\\ 0 &
    0 \end{array}} \leq \mnorm{x_{k l}+iy_{k l}}$, and hence
$\norm{\theta_2(x)}\leq 1$. In fact $\norm{\theta_2(x)}$ is equal to 1, for instance if $[z_{kl}]=I_{M_2(\mathbb{R})}$, then  $\norm{\theta_2(x)}\geq 1$. Thus $\norm{\theta_2(x)}= 1
\leq \sqrt{2}=\norm{x}$. 
\end{proof}

\begin{xrem}
 We end this section with a list of several results from the operator space theory which can be generalized for
the real operator spaces using the exact same proof as in the complex setting. Various constructions  
using real operator spaces like taking the quotient, infinite direct sums, $c_0$-direct sums, mapping spaces $CB(X,Y)$ and
matrix spaces $\mathbb{M}_{I,J}(X)$ can be defined analogously, and are real operator spaces. All the results and properties of matrix spaces hold true for the real operator spaces (see e.g. \cite[1.2.26]{BLM}).  Further, we can define
Hilbert row and Hilbert column operator space structure on a real Hilbert space by replacing $\mathbb{C}$ with $\mathbb{R}$, in the usual definition. Then $B(H,K)\cong CB(H^c, K^c)$ and $B(H,K)\cong CB(K^r, H^r)$ completely isometrically, for real Hilbert spaces $H,K$. Also $(H^c)^*\cong H^r$ and $(H^r)^*\cong H^c$. 
We can show that if $u: X\longrightarrow Z$ is completely bounded between real operator spaces $X$ and $Z$, and $Y$ is any subspace of 
$\Ker{u}$, then the canonical map $\tilde{u}: X/Y\longrightarrow Z$ induced by $u$ is completely bounded. If $Y=\Ker{u}$ then $u$ is a complete
quotient if and only if $\tilde{u}$ is a completely isometric isomorphism. The duality of subspaces and quotients hold in the real case, i.e., $X^*\cong Y^*/X^{\perp}$
and $(Y/X)^*\cong X^{\perp}$ completely isometrically, where $Y$ is a subspace of the real operator space $X$. It is also true that the trace class operator $S^1(H)$ is the predual of $B(H)$ for
every real Hilbert space $H$. If $X$ is a real operator space then $M_{m,n}(X)^{**}\cong M_{m,n}(X^{**})$ completely isometrically for all $m,n\in \mathbb{N}$. If $X$ and $Y$ are real operator spaces and if $u:X\longrightarrow Y^*$ is completely bounded, then its (unique) $w^*$-extension 
$\tilde{u}:X^{**}\longrightarrow Y^*$ is completely bounded with $\norm{\tilde{u}}=\norm{u}$. Hence $CB(X,Y^*)=w^*CB(X^{**},Y^*)$ completely  isometrically.
\end{xrem}

\section{Real Operator Algebras}

\begin{definition}
  An (abstract) real operator algebra $A$ is an algebra which is also an operator space, such that $A$ is completely isometrically
  isomorphic to a subalgebra of $B(H)$ for some (real) Hilbert space $H$, i.e., there exists a (real) completely
  isometric homomorphism $\pi :A\longrightarrow B(H)$.
For any $n$, $M_n(A)\subset M_n(B(H))=B(H^n)$ is a real operator algebra
with product of two elements, $[a_{ij}]$ and $[b_{ij}]$ of $M_n(A)$, given by
$$ [a_{ij}][b_{ij}]=[\sum_{k=1}^n a_{ik}b_{kj}].$$
\end{definition}

Every real operator algebra can be embedded (uniquely up to a complete
isometry) into a complex operator algebra via Ruan's `reasonable' complexification.
Let $A$ be a real operator algebra and $A_c=A+iA$
be the operator space complexification of $A$. Then $A_c$ is an algebra with a natural product
$$(x+iy)(v+iw)=(xv-yw)+i(xw+yv).$$

Suppose that 
$\pi :A\longrightarrow B(H)$ is a complete isometric homomorphism, 
for some real Hilbert space $H$. Then $\pi_c:A_c\longrightarrow B(H)_c$ is a (complex) complete isometry, and 
it is easy to see
that $\pi_c$ is also a homomorphism. Thus $A_c$ is a complex operator algebra if $A$ is a real operator algebra.
As in {\cite{RuaC}}, $B(H_c)=B(H)+iB(H)$ has a reasonable norm extension 
$\{\norm{.}_n\}_{n\in \mathbb{N}}$, these norms are inherited by $A_c$ via 
the complete isometric homomorphism $\pi_c :A_c\longrightarrow B(H)_c$. Thus the matrix norms on $A_c$ satisfy
$\norm{x+i0}_n=\norm{x}_n$ and $\norm{x+iy}=\norm{x-iy},$
for all $x+iy\in M_n(X_c)=M_n(X)+iM_n(X)$ and $n\in \mathbb{N}$.
The
conjugation ``-'' on $A_c$ satisfies $\overline{xy}=\bar{x}\bar{y}$, for all $x,y\in A_c$.

{\bf Remarks.}
1) The complexification of a real operator algebra is unique, up to
 complete isometry by \cite[Theorem 3.1]{RuaC}.

2) If $A$ is approximately unital then so is $A_c$. Indeed if $e_t$ is an approximate
unit for $A$, then for any $x+iy\in A_c$, 
$\norm{e_t(x+iy)-(x+iy)}\leq \norm{e_tx-x}+\norm{e_ty-y}$. Thus $e_t$ is an approximate
identity for $A_c$.

Now we show that there is a real version of the BRS theorem which characterizes the real operator algebras.
\begin{theorem}{(BRS Real Version)}
  Let $A$ be a real operator space which is also an approximately unital Banach space. Then the following are equivalent:
\begin{enumerate}[\upshape (i)]
\item The multiplication map $m: A\otimes_h A \longrightarrow A$ is completely contractive.
\item For any $n$, $M_n(A)$ is a Banach algebra. That is,
$$\mnorm{\sum_{k=1}^{n}{a_{{i}{k}}b_{{k}{j}}}}_{M_{n}(A)}\leq \mnorm{a_{{i}{j}}}_{M_{n}(A)} \mnorm{b_{{i}{j}}}_{M_{n}(A)},$$
for any $[a_{{i}{j}}]$ and $[b_{{i}{j}}]$ in $M_{n}(A).$
\item $A$ is a real operator algebra, that is, there exist a real
Hilbert space $H$ and a completely isometric homomorphism
$\pi:A\longrightarrow B(H)$.
\end{enumerate}
\end{theorem}
\begin{proof}
The equivalence between ${\rm (i)}$ and ${\rm (ii)}$, and that ${\rm (iii)}$ implies these, follows from the property that the Haagerup tensor product of real operator spaces linearizes completely bounded bilinear maps, and the fact that each $M_n(A)$ is an operator algebra.

  ${\rm (iii)} \Rightarrow {\rm (ii)}$ \  Let $\pi: A\longrightarrow B(H)$. Then by {\cite[Theorem 2.1]{RuaC}}, 
$\pi_c:A_c\longrightarrow B(H)_c$ is a complete isometric homomorphism. 
Let $[a_{ij}],[b_{ij}] \in M_n(A)\subset M_n(A_c)$. Then by the BRS theorem for complex operator algebras,
$$\mnorm{\sum_{k=1}^{n}{a_{{i}{k}}b_{{k}{j}}}}_{M_{n}(A)}\leq \mnorm{a_{{i}{j}}}_{M_{n}(A)} \mnorm{b_{{i}{j}}}_{M_{n}(A)}.$$

${\rm (ii)} \Rightarrow {\rm (iii)}$ \  Since $A$ is approximately unital, by the above remark, $A_c$ is also approximately unital.
 Let $\theta:A_c\longrightarrow
  M_2(A)$ be $\theta(x+iy)= \left[ \begin{array}{cc} x & y \\ -y &
      x \end{array} \right]$. Then 
 $\theta$ is a complete isometric homomorphism and each amplification, $\theta_n$ is an isometric homomorphism.
Let 
$a=[a_{ij}], \ b=[b_{ij}]\in M_n(A_c)$. Then
$$\norm{ab}=\norm{\theta_n(ab)}=\norm{\theta_n(a)\theta_n(b)}\leq \norm{\theta_n(a)}\norm{\theta_n(b)}=
\norm{a}\norm{b}.$$
Thus by the BRS theorem for complex operator algebras, there exists a completely isometric homomorphism 
 $\pi: A_c\longrightarrow B(K)$, for some complex Hilbert space $K$. Let $K=H_c$. 
Define 
$\pi_1=\frac{\pi+\overline{\pi}}{2}$ and $\pi_2=\frac{\pi-\overline{\pi}}{2i}$. Then $\pi_1, \ \pi_2$ are (complex) linear maps such that $\pi_1=\overline{\pi_1}$, $\pi_2=\overline{\pi_2}$, and
 $\pi(x+iy)=(\pi_1(x)-\pi_2(y))+i(\pi_1(y)+\pi_2(x))$. Let $\tilde{\pi}$ be the composition
of $\pi$ with the canonical identification $B(K)\hookrightarrow M_2(B(H))$ (see e.g. (\ref{eq2})), so
$$\tilde{\pi}(x+iy)=\left [ \begin{array}{cc}
                            \pi_1(x)-\pi_2(y) & -\pi_1(y)-\pi_2(x) \\
                            \pi_1(y)+\pi_2(x) &  \pi_1(x)-\pi_2(y)
                           \end{array} \right ]
\in M_2(B(H)).$$
The restriction of $\tilde{\pi}$ to $A$, say $\pi_{\circ}$, is a complete isometric inclusion from $A$ into $M_2(B(H))$. Also, for $x,v\in A$
\begin{eqnarray*}
\pi_{\circ}(x) \pi_{\circ}(v) &=& \left [ \begin{array}{cc}
                            \pi_1(x) & -\pi_2(x) \\
                            \pi_2(x) &  \pi_1(x)
                           \end{array} \right ]
                            \left [ \begin{array}{cc}
                            \pi_1(v) & -\pi_2(v) \\
                            \pi_2(v) &  \pi_1(v)
                           \end{array} \right ]\\
                           &=& \left [ \begin{array}{cc}
                            \pi_1(x)\pi_1(v)-\pi_2(x)\pi_2(v) & -\pi_1(x)\pi_2(v)-\pi_2(x)\pi_1(v) \\
                             \pi_1(x)\pi_2(v)+\pi_2(x)\pi_1(v) &  \pi_1(x)\pi_1(v)-\pi_2(x)\pi_2(v)
                           \end{array} \right ]\\
                           &=& \left [ \begin{array}{cc}
                            \pi_1(xv) & -\pi_2(xv) \\
                            \pi_2(xv) &  \pi_1(xv)
                           \end{array} \right ] = \pi_{\circ}(xv)
\end{eqnarray*}
Thus $\pi_{\circ}$ is a completely isometric homomorphism from $A$ into $M_2(B(H))\cong B(H^2)$.
\end{proof} 

\begin{theorem}
Let $A$ be a complex operator algebra. Then $A$ is a complexification of a real operator algebra $B$, i.e.,
$A=B_c$ completely isometrically if and only if there exists a complex conjugation ``$-$'' on $A$ such that
\begin{enumerate}[\upshape (i)]
\item ``$-$'' is a complete isometry, i.e., $\mnorm{x_{ij}}_n=\mnorm{\overline{x_{ij}}}_n$ for all $[x_{ij}]\in M_n(A)$ and $n\in \mathbb{N}$,
\item $\overline{xy}=\bar{x} \bar{y}$ for all $x,y\in A$. 
\end{enumerate}
\end{theorem}
\begin{proof}
If $A=B_c$, for a real operator algebra $B$, then clearly $A$ satisfies the conditions in (i) and (ii) above.
Suppose that $A$ is a complex operator algebra such that (i) and (ii) hold. Since $A$ is a complex operator space such that the matrix norms satisfy (i), by \cite[Theorem 3.2]{RuaC} there exists a real operator space $B$ such that $A=B+iB$ completely
isometrically. Now the conjugation on $A$ is $\overline{x+iy}=x-iy$, and $B={\rm Re}(A)=\{x\in A: x=\bar{x}\}$. So if $x, y\in B$, then $xy=\bar{x}\bar{y}=\overline{xy}$. Thus $B$ is a subalgebra. Since $A$ is a complex operator algebra, it is also a real operator algebra, and $B$ is a (real) closed subalgebra of $A$. Thus $B$ is a (real) 
operator algebra. 
\end{proof}

Let $A\subset B(H)$ be a real operator algebra, for some real Hilbert space $H$. Define the unitization
 of $A$ as $A^1={\rm Span}_{
  \mathbb{R}}\{ A,I_H \}\subset B(H)$. Then $A\subset A^1\subset B(H)$ is a
closed subalgebra, and $A^1$ is a unital real operator algebra.

\begin{lemma}
 Let $A\subset B(H)$ be a real operator algebra. Then $(A_c)^1=(A^1)_c\subset B(H)_c$, completely isometrically.
\end{lemma}
\begin{proof}
Clearly both $(A_c)^1$ and $(A^1)_c$ are subsets of $B(H)_c$. 
Since $A\subset A^1$, $A_c\subset (A^1)_c$ and $I_{H}\in (A^1)_c$. So 
$(A_c)^1={\rm Span}\{A_c, I_H \}\subset (A^1)_c$. If $x \in (A^1)_c$, then 
\begin{eqnarray*}
 x &=&( \alpha a+  \alpha^{'} I_H) +i( \beta b +  \beta^{'} I_H)\\
&=&( \alpha a+i \beta b)+( \alpha^{'}+ i\beta^{'}) I_H \ \in \ {\rm Span}\{A_c +I_H\}\subset (A_c)^1.  
\end{eqnarray*}
Thus $(A_c)^1=(A^1)_c$.
\end{proof}
The following result shows that the unitization of real operator algebras is independent of the
choice of the Hilbert space $H$.
\begin{theorem}{(Real Version of Meyer's Theorem)}  
  Let $A\subseteq B(H)$ be a real operator algebra, and suppose that $I_H\notin
  A$. Let $\pi:A\longrightarrow B(K)$ be a
   completely contractive homomorphism, where $K$ is a
  real Hilbert space. We 
  extend $\pi$ to $\pi^{\circ}:A^1\longrightarrow B(K)$ by
  $\pi^{\circ}(a+\lambda I_H)=\pi(a)+\lambda I_K$, $a\in A$, $\lambda
  \in \mathbb{C}$. Then $\pi^{\circ}$ is a 
  completely contractive homomorphism.
\end{theorem}
\begin{proof}
  Consider $\pi_c:A_c\longrightarrow B(K)_c\cong B(K_c),$ which is a
   completely contractive homomorphism. Now extend
  $\pi_c$ to $(\pi_c)^{\circ}:( A_c)^{1}\longrightarrow B(K_c)$ by
  $(\pi_c)^{\circ}(a+\lambda I_{H_c})=\pi_c(a)+\lambda I_{K_c}$, 
$a\in
  A_c$, $\lambda \in \mathbb{C}$. Then by the Meyer's Theorem for complex
  operator algebras (\cite[Corollary 2.1.15]{BLM}), $(\pi_c)^{\circ}$ is a  
  completely contractive homomorphism. 
Let $a+\lambda I_H\in A^1$, then $(\pi_c)^{\circ}(a+\lambda I_H)=\pi_c(a)+\lambda I_K=\pi(a)+\lambda I_K=\pi^{\circ}(a+\lambda I_H)$. Thus $(\pi_c)^{\circ}|_{A^1}=\pi^{\circ}$ and hence $\pi^{\circ}$ is 
a completely contractive homomorphism.
\end{proof}

\section{Real Injective Envelope}\label{rie}
In this section we study in more detail the real injective envelope of real operator spaces, which is mentioned by Ruan in \cite{RuaR}.  

\begin{definition}
  Let $X$ be a real operator space and let $Y$ be a real 
  operator space, such that there is a complete isometry
  $i:X\longrightarrow Y$. Then the pair $(Y,i)$ is called an 
  extension of $X$. An injective extension $(Y,i)$ is a {\em real injective
  envelope} of $X$ if there is no real injective space Z such that
  $i(X)\subset Z\subset Y$. We denote a real injective envelope
  by $(I(X), i)$ or simply by $I(X)$.
\end{definition}

By the Arveson-Wittstock-Hahn-Banach theorem for real operator spaces, {\cite[Theorem 3.1]{RuaR}}, 
$B(H)$ is an injective real operator space for any real Hilbert space $H$. Thus
a real operator
space $X\subset B(H)$ is injective if and only if it is the range of a completely contractive idempotent map from
$B(H)$ onto $X$.

\begin{definition}
If $(Y,i)$ is an extension of $X$, then $Y$ is a $rigid$
extension if $I_Y$ is the only completely contractive map which
restricts to an identity map on $X$. We say that $(Y,i)$ is an $essential$
extension of $X$, if whenever $u:X\longrightarrow Z$ is a completely
contractive map, for some real operator space $Z$, such that $u\circ i$ is a complete isometry, then $u$
is a complete isometry. 
\end{definition}

\begin{theorem}\label{realIE}
If a real operator $X$ is contained in a real injective operator space $W$, then there is an injective envelope
$Y$ of $X$ such that $X\subset Y\subset W$.
\end{theorem}

To prove this theorem we need to define some more terminology, and we also need the following two lemmas, which are the real analogies of \cite[Lemma 4.2.2]{BLM} and
\cite[Lemma 4.2.4]{BLM}, respectively. The proof of Lemma \ref{Xproj} uses the fact that, if $X$ is a real operator space and $H$ is any real Hilbert space, then a bounded net $(u_t)$ in $CB(X,B(H))$ converges 
in ${\rm weak}^*$-topology to a $u\in CB(X,B(H))$ if and only if 
$$\ip{u_t(x)\zeta}{\eta}\to \ip{u(x)\zeta}{\eta}\ {\rm for\ all\ } x\in X,\; \zeta, \eta \in H.$$
\begin{definition}
Let $X$ is a subspace of a real operator space $W$. An {\em $X$-projection} on $W$ is a completely contractive (real) idempotent map $\phi: W\to W$ which restricts to the identity map on $X$. An {\em $X$-seminorm} on $W$ is a seminorm of the form $p(\cdot)=\norm{u(\cdot)}$, for a completely contractive (real) linear map $u:W \to W$ which restricts
to the identity map on $X$.
Define a partial order $\leq$ on the sets of all $X$-projections, by setting $\phi\leq \psi$ if 
$\phi \circ \psi=\psi \circ \phi=\phi$. This is also equivalent to $\Ran{\phi}\subset \Ran{\psi}$ and 
$\Ker{\psi}\subset \Ker{\phi}$.
\end{definition}

\begin{lemma}\label{Xproj}
 Let $X$ be a subspace of a real injective operator space $W$.
\begin{enumerate}[\upshape (i)]
 \item Any decreasing net of $X$-seminorms on $W$ has a lower bound. Hence there exists a minimal $X$-seminorm on $W$, by Zorn's lemma. Each $X$-seminorm majorizes a minimal $X$-seminorm.
\item If $p$ is a minimal $X$-seminorm on $W$, and if $p(\cdot)=\norm{u(\cdot)}$, for a completely
contractive linear map on $W$ which restricts to the identity map on $X$, then $u$ is a minimal $X$-projection.
\end{enumerate}
 \end{lemma}

\begin{lemma}\label{IEc}
Let $(Y,i)$ be an extension of real operator space $X$ such that $Y$ is injective. Then the following are equivalent:
\begin{enumerate}[\upshape (i)]
\item $Y$ is an injective envelope of $X$,
\item $Y$ is a rigid extension of $X$,
\item $Y$ is an essential extension of $X$.
\end{enumerate}
\end{lemma}

Using the rigidity property of injective envelopes and a standard diagram chase, we can show that if $(Y_1,i_1)$ and $(Y_2,i_2)$ are two injective envelopes of a real operator space $X$ then $Y_1$ and $Y_2$ are completely isometrically isomorphic via some map $u$ such that $u\circ i_1=i_2$. Hence the real injective envelope, if exists,
is unique.  The argument in \cite[Theorem 4.2.6]{BLM}, and Lemma \ref{Xproj} and Lemma \ref{IEc}, prove Theorem \ref{realIE}. Thus
the real injective envelope exists.

\begin{lemma}\label{lemmainj}
  Let $X$ be a real operator space. Then $X$
  is real injective iff $X_{c}$ is (complex) injective.
\end{lemma}
\begin{proof}
  First suppose that $X$ is real injective, then there exists a completely contractive
  idempotent
  $P$, from $B(H)$ onto $Z$, for some real Hilbert space $H$. The
  complexification of $P$, $P_{c}:B(H)_{c}\longrightarrow X_{c}$ is clearly a 
(complex) completely contractive idempotent onto $X_c$. 
Since $B(H)_{c}\cong B(H_{c})$, completely isometrically, $Z_{c}$ is a (complex) injective operator space.
 Conversely, let $X_{c}$ be a (complex) injective space and
  $Q:B(K)\longrightarrow X_{c}$ be a completely contractive (complex) linear surjective idempotent. 
Let $K=H_{c}$ where $H$ is a real Hilbert space, 
so $Q:B(H_{c})\cong
  B(H)_{c}\longrightarrow X_{c}$. 
Consider ${\rm Re}(Q)=\frac{Q+\bar{Q}}{2}$, where
  $\bar{Q}(T+iS)=\overline{Q(T-iS)}$. For any $T+iS\in B(H)_c$, 
$$\overline{Q}^2(T+iS)=\overline{Q}(\overline{Q(T-iS)})=\overline{Q^2(T-iS)}=\overline{Q(T-iS)}=\overline{Q}(T+iS).$$
Let $x+iy\in X_c$ and suppose that $Q(T+iS)=x-iy$ for some $T,S\in B(H)$. Then $\overline{Q}(T-iS)=x+iy$. Thus $\overline{Q}$ is an idempotent onto $X_c$. 
So $\overline{Q}Q(T+iS)=Q(T+iS)$ and $Q\overline{Q}(T+iS)=\overline{Q}(T+iS)$, for all $T+iS\in B(H)_c$. Thus for $T\in B(H)$, 
$({\rm Re}(Q))^2(T)=\frac{Q^2(T)+Q\overline{Q}(T)+\overline{Q}Q(T)+\overline{Q}^2(T)}{4}=\frac{2Q(T)+2\overline{Q}(T)}{4}={\rm Re}(Q)(T)$. If $x\in X\subset X_c$, then $Q(x)=x$ and $\overline{Q}(x)=x$, so ${\rm Re}(Q)(x)=x$. This shows that  
 ${\rm Re}(Q):B(H)\longrightarrow X$ is a
  (real) linear completely contractive idempotent onto $X$. Hence $X$ is real injective.
\end{proof}

The next result is a real analogy of a Choi-Effros theorem (see e.g., \cite[Theorem 1.3.13]{BLM}). It is shown in the last paragraph of \cite[pg. 492]{RuaR}) that the argument in the complex version of the theorem can be reproduced to prove part (i) of the following result. 

\begin{theorem}[Choi-Effros]{\label{ChoiEff}}
  Let $A$ be a unital real $C^*$-algebra and let $\phi: A\longrightarrow A$
  be a selfadjoint, completely positive, unital, idempotent
  map. Then
  \begin{enumerate}[\upshape (i)]
  \item $R=\Ran{\phi}$ is a unital real $C^*$-algebra with respect to the
  original norm, involution, and vector space structure, but new
  product $r_1\circ_{\phi}r_2=\phi(r_1r_2)$,
  \item $\phi(ar)=\phi(\phi(a)r)$ and $\phi(ra)=\phi(r\phi(a))$, for
  $r\in R$ and $a\in A$,
  \item If $B$ is the $C^*$-algebra generated by the set $R$, and if $R$ is given the product $\circ_{\phi}$, then $\phi\mid_B$ is a $*$-homomorphism from $B$ onto $R$.
\end{enumerate}
\end{theorem}
\begin{proof}
Let $\phi: A\longrightarrow A$ be a selfadjoint, completely positive, unital idempotent map. Then $\phi$ is 
completely contractive, by \cite[Proposition 4.1]{RuaR}, and hence $\phi_c:A_c\longrightarrow A_c$ is a completely
contractive, unital idempotent onto $\Ran{\phi}_c$. By the Choi-Effros Lemma for complex operator systems, \cite[Theorem 1.3.13]{BLM}, 
 $\Ran{\phi}_c$ is a $C^*$-algebra with a new product given by 
$(r_1+ir_2)\circ (s_1+is_2)=\phi_c((r_1+ir_2) (s_1+is_2))$, $r_1,r_2,s_1,s_2\in R$. For $r,s\in R$, $r\circ s=\phi_c(rs)=\phi(rs)\in R$. By \cite[Proposition 5.1.3]{Li}, $R$ is a real $C^*$-algebra with this product. Further, 
$\phi(ar)=\phi_c(ar)=\phi_c(\phi_c(a)r)=\phi(\phi(a)r)$, and similarly $\phi(ra)=\phi(r\phi(a))$, for all $a\in A, r\in R$. 

Let $C=C^*(R_c)$ be the (complex) $C^*$-algebra generated by $R_c$ in $A_c$, then by \cite[Theorem 1.3.13 (iii)]{BLM}, $(\phi_c) |_C$ is a $*$-homomorphism
from $C$ onto $R_c$. Let $B=C^*(R)$ be the real $C^*$-subalgebra of $A$ generated by $R$. 
It is easy to see that $C^*(R_c)=C^*(R)_c$. Clearly, since $C^*(R)\subset C^*(R_c)$, $C^*(R)_c\subset C^*(R_c)$. 
Also, 
$${\rm \bf S}_{\mathbb{C}}\{ s_1s_2\hdots s_n\; :\; n\in \mathbb{N}\}={\rm \bf S}_{\mathbb{R}}\{r_1r_2\hdots r_n\; :\; n\in \mathbb{N}\}+i {\rm \bf S}_{\mathbb{R}}\{r_1^{'}r_2^{'}\hdots r_n^{'}\; :\; n\in \mathbb{N}\},$$ 
where $s_i\in R_c$,   and $r_i, r_i^{'} \in R$, and ${\rm \bf S}$ means ``Span''.
If
$a\in C^*(R_c)\subset A_c $ then $a=x+iy$ is the limit of $a_t\in {\rm \bf S}_{\mathbb{C}}\{ s_1s_2\hdots s_n\;  :\; n\in \mathbb{N}\}$. Then $a_t=x_t+iy_t$, where $x_t\in {\rm \bf S}_{\mathbb{R}}\{r_1r_2\hdots r_n\; :\; n\in \mathbb{N}\}$,
$y_t\in {\rm \bf S}_{\mathbb{R}}\{r_1^{'}r_2^{'}\hdots r_n^{'}\; :\; n\in \mathbb{N}\}$. Also, if we suppose that $A_c\subset B(H)_c$, for some real Hilbert space $H$, then it is easy to see that $x_t\longrightarrow x$, $y_t\longrightarrow y$. 
Hence, $(\phi_c)_B=\phi |_B$ is a $*$-homomorphism
from $B$ onto $R$. 
\end{proof}

\begin{xrem}
 Let $A$ and $B$ be real $C^*$-algebras, and let $\phi: A\longrightarrow B$ be a unital completely contractive map. Then $\phi_c$ is a (complex) 
completely contractive linear map between complex $C^*$-algebras $A_c$ and $B_c$. So $\phi_c$ is completely positive and hence selfadjoint. Since $\phi=\phi_c|_A$, $\phi$ is also selfadjoint. Thus a completely contractive unital map between real $C^*$-algebras is selfadjoint.     
As a result, we can replace the completely positive
  and selfadjoint condition in Theorem \ref{ChoiEff} above, with the condition
  that $\phi$ is completely contractive.
\end{xrem}
\begin{theorem}\label{IE}
  $X$ be a unital real operator space, then there is an injective envelope
  $I(X)$ which is a unital real $C^*$-algebra.
\end{theorem}
\begin{proof}
  Let $X\subset B(H)$ for some real Hilbert space $H$. Since $B(H)$
  is injective, we can find an injective envelope of $X$ such that
  $X\subset{I(X)\subset B(H)}$. As
  $I(X)$ is injective, so the identity map on
  $I(X)$ extends to $\phi: B(H)\longrightarrow
  B(H)$ such that $\phi$ is a completely contractive idempotent
  onto $I(X)$. By Theorem \ref{ChoiEff} and the remark above, 
$\Ran{\phi}=I(X)$ becomes a unital real
  $C^*$-algebra with the new product.
\end{proof}

\begin{proposition} 
  Let $X$ be a real (or complex) Banach space, then 
  $\Min{I(X)}=I(\Min{X})$, completely isometrically.
\end{proposition}
\begin{proof}
  Let $X$ be a real Banach space. Since $I(X)$ is an injective Banach space, and contractive maps into $\Min{X}$ are completely contractive, it clear that  $\Min{I(X)}$ is a real injective operator space. 
Let $i:X\longrightarrow I(X)$ be the canonical isometry,  and let $j:I(X)\longrightarrow C(\Omega,\mathbb{R})$ be an isometric embedding of $I(X)$, for some compact, Hausdorff space $\Omega$. Then $j:\Min{I(X)}\longrightarrow C(\Omega,\mathbb{R})$ and  $j\circ i:\Min{X}\longrightarrow C(\Omega,\mathbb{R})$ are complete isometries. Thus $(\Min {I(X)}, i)$ is a real injective extension of $\Min{X}$. Further suppose that $u:\Min{I(X)}\longrightarrow \Min{I(X)}$ is a complete contraction which restricts to the identity map on $\Min{X}$. Then  by the rigidity of $I(X)$, $u$ is an isometry into $\Min{I(X)}$, and hence a complete isometry. Thus $(\Min{I(X)}, i)$ is a rigid extension of $\Min{X}$, and hence $I(\Min{X})=\Min{(I(X))}$, completely isometrically.
\end{proof}

\begin{definition}
  Let $X$ be a real unital operator space. Then we define a
  {\em $C^*$-extension} of $X$ to be a pair $(B,j)$ consisting of a unital
  real $C^*$-algebra $B$, and a complete isometry $j:X\longrightarrow
  B$, such that $j(X)$ generates $B$ as a $C^*$-algebra.
  A $C^*$-extension $(B,i)$ is a {\em $C^*$-envelope} of
  $X$ if it has the the following
  universal property: Given any $C^*$-extension $(A,j)$ of $X$, there
  exists a (necessarily unique and surjective) real $*$-homomorphism
  $\pi:A\longrightarrow B$, such that $\pi\circ j=i$.
\end{definition}

Using Theorem \ref{ChoiEff}, Theorem \ref{IE}, and the argument in \cite[4.3.3]{BLM}, we can show that
 the $C^*$-subalgebra of $I(X)$ generated by 
$i(X)$ is a $C^*$-envelope of $X$, where the pair $(I(X), i)$ is an injective envelope of $X$.
Thus the $C^*$-envelope exists for every unital real operator space $X$.

\medskip

A {\em real operator system} is a (closed) subspace $\mathcal{S}$ of B(H), $H$ a real Hilbert space, such that $\mathcal{S}$ contains $I_H$, and $\mathcal{S}$ is selfadjoint, i.e., $x^* \in \mathcal{S}$ if and only if $x \in \mathcal{S}$. 
Note that a positive element in $B(H)$, $H$ a real Hilbert space, need not be selfadjoint. 
For instance, consider the $2\times 2$ matrices over $\mathbb{R}$, then $x=\left[ \begin{array}{cc} 2 & -1 \\ 1 & 2 \end{array} \right] $ is positive, i.e., $\ip{x\zeta}{\zeta}\geq 0$ for all $\zeta\in \mathbb{R}^2$, but $x\neq x^*$. Thus, we say that an element $x\in \mathcal{S}(X)\subset B(H)$ is {\em positive}, if for all $\zeta, \eta\in H$,
$\ip{x\zeta}{\eta}=\ip{\zeta}{x\eta}$ (selfadjoint), and $\ip{x\zeta}{\zeta}\geq 0$. 
 If $x\in B(H,K)$, $H$ and $K$ real Hilbert spaces, then
\begin{equation}\label{eqn1}
\left[ \begin{array}{cc} 1 & x \\ x^* & 1 \end{array} \right]\geq 0 \;\; \iff  \norm{x}\leq 1.
\end{equation}
In \cite{RuaR}, Ruan considers real operator systems and shows that a unital selfadjoint map between two real operator systems is completely contractive if and only if it is completely positive.  It is also shown that the Stinespring theorem, the Arveson's Extension Theorem, and the Kadison-Schwarz inequality hold true, with an added hypothesis that the maps be selfadjoint. 
We can show using the  Stinespring theorem that  Proposition 1.3.11 and Proposition 1.3.12 from \cite{BLM}, are also true in the real setting.

If $X\subset B(H)$ is a real operator space, then we can define the {\em real Paulsen system} as
$$\mathcal{S}(X)=\left [ \begin{array}{cc} 
                         \mathbb{R} I_H & X \\
                        X^{\star} & \mathbb{R} I_H 
                         \end{array}
                  \right]
=           \left \{ \left [ \begin{array}{cc}
                          \lambda & x \\
                            y^* & \mu
                         \end{array}
                  \right]\ :\ x,y \in X\ {\rm and}\ \lambda, \mu \in \mathbb{R} \right \} \subset M_2(B(H)).
$$
The next lemma is the real version of Paulsen lemma, and it can be proved using the argument in \cite[Lemma 1.3.15]{BLM}, Equation (\ref{eqn1}), and that the map $\phi$, defined below, is selfadjoint. This lemma shows that as a real operator system (i.e., up to complete order isomorphism) $\mathcal{S}(X)$ only depends on the operator space structure of $X$, and not on its representation on $H$.
\begin{lemma}\label{Pau}
 For $i=1,2$, let $H_i$ and $K_i$ be real Hilbert spaces, and $X_i\subset B(K_i,H_i)$. Suppose that $u:X_1\to X_2$
is a real linear map. Let $\mathcal{S}_i$ be the real Paulsen systems associated with $X_i$ inside $B(H_i\oplus K_i)$. If $u$ is contractive (resp.\ completely contractive, completely isometric), then
$$\phi: \left [ \begin{array}{cc}
                          \lambda & x \\
                            y^* & \mu
                         \end{array}
                  \right]\
\to 
\left [ \begin{array}{cc}
                          \lambda & u(x) \\
                            u(y)^* & \mu
                         \end{array}
                  \right]  $$
is positive (resp.\ completely positive and completely contractive, a complete order injection) as a map from $\mathcal{S}_1$ to $\mathcal{S}_2$.
\end{lemma}

Let $X\subset B(H)$ be a real operator space and let $\mathcal{S}(X)\subset M_2(B(H))$ be the associated real Paulsen system. Then  $ I(\mathcal{S}(X))\subset M_2(B(H))$ is a unital $C^*$-algebra, by Theorem \ref{IE}, and there is a completely positive idempotent map $\phi$ from $M_2(B(H))$ onto $I(\mathcal{S}(X))$.
Let $p$ and $q$ be the canonical projections $I_H\oplus 0$ and $0\oplus I_H$, then $\phi(p)=p$ and $\phi(q)=q$. So, $$I(\mathcal{S}(X))=\left [ \begin{array}{cc}
                          pI(\mathcal{S}(X))p & pI(\mathcal{S}(X))q\\
                           qI(\mathcal{S}(X))p & qI(\mathcal{S}(X))q
                         \end{array}
                  \right]\ .$$ 
Using Lemma \ref{IEc} and Lemma \ref{Pau}, and the argument in \cite[Theorem 4.4.3]{BLM}, we can show that the 1-2-corner,
$pI(\mathcal{S}(X))q$, of $I(\mathcal{S}(X))$ is an injective envelope of $X$. As a corollary, we get the following which is the real analogue of the Hamana-Ruan characterization of injective operator spaces.

\begin{theorem}
 A real operator space $X$ is injective if and only if $X\cong pA(1-p)$ completely isometrically, for a projection $p$ in an injective real $C^*$-algebra $A$.
\end{theorem}

 A {\rm real TRO} is a closed linear subspace $Z$ of $B(K,H)$, for some real Hilbert spaces $K$ and $H$,
satisfying $ZZ^{\star}Z\subset Z$. For $x,y,z \in Z$, $xy^*z$ is called the {\em triple or ternary product} on $Z$, sometimes written as
$[x, y, z]$. A {\em subtriple} of a TRO $Z$ is a closed subspace $Y$ of $Z$ satisfying $YY^{\star}Y\subset Y$. A {\em triple morphism} between TROs is a linear map which respects the triple product: thus $T([x, y, z])=[Tx, Ty, Tz]$. In the construction
of the real injective envelope, discussed above, let $Z=pI(\mathcal{S}(X))q$, then $ZZ^{\star}Z\subset Z$ with the product of the $C^*$-algebra $I(\mathcal{S}(X))$. In terms of the product in $B(H)$, $[x,y,z]=P(xy^*z)$ for $x,y,z\in Z$. So if $X$ is a TRO, then the triple product on $X$ coincides with the triple product on $X$ coming from $I(X)$. Thus $pI(\mathcal{S}(X))q=I(X)$ is a TRO. If two TROs $X$ and $Y$ are completely isometrically isomorphic, via say $u$, then by Lemma \ref{Pau}, we can extend $u$ to a complete order isomorphism between the Paulsen systems. Further, this map extends to a completely isometric unital surjection $\tilde{u}$ between
 the the injective envelopes $I(\mathcal{S}(X))$ and $I(\mathcal{S}(Y))$, which are (real) unital $C^*$-algebras. By Lemma
 \ref{lemma1},  $\tilde{u}$ is a $*$-isomorphism, and hence a ternary isomorphism between when restricted to $X$. Thus $u$ is a triple isomorphism. Thus a real operator space can have at most one triple product (up to complete isometry).

Define $T(X)$ to be the smallest subtriple of $I(X)$ containing $X$. Then it is easy to see that
$$T(X)=\overline{\rm Span}\{ x_1x_2^*x_3x_4^*\hdots x_{2n+1}\ :\ x_1,x_2,\hdots x_{2n+1}\in X \}.$$
Let $B=T(X)^{\star}T(X)$, $T(X)$ regarded as a subtriple of $I(X)$ in $I(\mathcal{S}(X))$. Then $B$ is a $C^*$-subalgebra
of $2$-$2$-corner of $I(\mathcal{S}(X))$, and hence of $I(\mathcal{S}(X))$. Define $\ip{y}{z}=y^*z$ for $y,z\in  T(X)$, a 
$B$-valued inner product. This inner product is called the {\em Shilov inner product} on $X$.

\section{One-Sided Real $M$-Ideals}

Let $X$ be a real operator space. If $P$ is a projection, i.e., $P=P^2$ and $P^*=P$ (equivalently $\norm{P}\leq 1$), then define linear mappings
$$\nu_P^c:X\longrightarrow C_2(X):x\mapsto \left [\begin{array}{c} P(x)\\ x-P(x)\end{array}\right ],$$
$$\mu_P^c:C_2(X)\longrightarrow X:\left [\begin{array}{c} x\\ y\end{array}\right ] \mapsto P(x)+({\rm Id}-P)(y).$$
Then $\mu_P^c\circ \nu_P^c=I$.
\begin{definition}
A {\em complete left $M$-projection} on $X$ is a linear idempotent on $X$ such that the map $\nu_P^c:X\longrightarrow  C_2(X):x\mapsto \left [\begin{array}{c} P(x)\\ x-P(x)\end{array}\right ]$ is a complete isometry.
\end{definition}

\begin{proposition}\label{charac1}
If $X$ is a real operator space and $P:X\longrightarrow X$ is a projection, then $P$ is a complete left $M$-projection if and only if $\mu_P^c$ and $\nu_P^c$ are both completely contractive. 
\end{proposition}
\begin{proof}
 If $\nu_P^c$ is completely isometric, then 
$$\norm{P(x)+y-P(y)}
=\mnorm{\begin{array}{c}
P(x)\\ y-P(x)
\end{array}
}
\leq \mnorm{\begin{array}{c}
P(x)\\
x-P(x)\\
P(y)\\
y-P(y)
\end{array}
}
=\mnorm{\begin{array}{c}
x\\
y
\end{array}
},
$$ 
and thus $\mu_P^c$ is contractive. These calculations work as well for matrices. The converse follows from the fact that $\mu_P^c\circ \nu_P^c=I$.
\end{proof}

\begin{proposition}\label{charac2}
 The complete left $M$-projections in a real operator space $X$ are just the mappings $P(x)=ex$ for a 
completely isometric embedding $X\hookrightarrow B(H)$ and an orthogonal projection $e\in B(H)$.
\end{proposition}
\begin{proof}
 If $P:X\longrightarrow X$ is a complete left $M$-projection, then fix an embedding $X\subset B(H)$ for some
real Hilbert space $H$. By the definition, the mapping 
$$\sigma:X\hookrightarrow B(H\oplus H): x\mapsto \left [ \begin{array}{cc}   P(x) & 0 \\
                                                                            (I-P)(x) & 0
                                                        \end{array} \right ]
$$
is completely isometric. We have that
$$ \sigma (P(x))=\left[ \begin{array}{cc} P(x) & 0 \\
                                          0 & 0
\end{array} \right]=
\left [ \begin{array}{cc} 1 & 0 \\
                                          0 & 0
\end{array} \right ]
\sigma(x),
$$
and thus $e=\left[ \begin{array}{cc} 1 & 0 \\
                                          0 & 0
\end{array} \right] \in B(H\oplus H)$ is the desired left projection relative to the embedding $\sigma$. The converse follows from the following:
\begin{eqnarray*}
 \mnorm{\begin{array}{c}
                 P(x) \\ x-P(x)
\end{array}}^2
&=&\mnorm{\begin{array}{c}
                 ex \\ x-ex
\end{array}}^2
=\norm{\left[ \begin{array}{cc}
x^*e & x^*-x^*e
\end{array} \right]
\left[ \begin{array}{c}
 ex \\ x-ex
\end{array}\right]
}\\
&=&\norm{x^*ex +x^*(1-e)x}
=\norm{x^*x}
=\norm{x}^2.
\end{eqnarray*}
\end{proof}

 Let $X$ be a real operator space. We say a map
$u:X\longrightarrow X$ is a {\em left multiplier} of $X$ if there
exists a linear complete isometry $\sigma: X\longrightarrow B(H)$ for 
some real Hilbert space $H$, and
an operator $S\in B(H)$ such that
$$\sigma(u(x))=S\sigma(x),$$
for all $x\in X$. We denote the set of all left multipliers of $X$ by
$\Ml(X)$. Define the multiplier norm of $u$, to be the
infimum of $\norm{S}$ over all such possible $H, S, \sigma$. 
We define a {\em left adjointable map} of $X$ to be
 a linear map $u:X\longrightarrow X$ such that there exists a linear
 complete isometry $\sigma: X\longrightarrow B(H)$ for some real Hilbert space $H$,
 and an operator
 $A\in B(H)$ such that
 $$\sigma(u(x))=A\sigma(x)\ {\rm for\  all\ } x\in X, \ {\rm and}\ A^*\sigma(X)\subset \sigma(X).$$
 The collection of all left adjointable maps of $X$ is denoted by
 $\Al(X)$. Every left adjointable map of $X$ is a left multiplier of
 $X$, that is, $\Al(X)\subset \Ml(X)$.

\begin{theorem}\label{CharMulti}
Let $X$ be a real operator space and let $u: X \longrightarrow X$ be a linear map. Then the 
following are equivalent:
\begin{enumerate}[\upshape (i)]
\item $u$ is a left multiplier of $X$ with norm $\leq 1$.
\item The map 
$\tau_u: C_2(X)\longrightarrow C_2(X): \left [\begin{array}{c} x \\ y \end{array} \right ]\mapsto \left [\begin{array}{c} u(x) \\ y \end{array} \right ]$, is completely contractive.
\item There exists a unique `$a$' in the $1$-$1$-corner of $I(\mathcal{S}(X))$ such that $\norm{a}\leq 1$ and 
$u(x)=ax$ for all $x\in X$. 
\end{enumerate}
\end{theorem}

 By a direct application of the argument in \cite[Theorem 4.5.2]{BLM}, we get that (i) $\Rightarrow $ (ii) and 
(iii) $\Rightarrow$ (i). For the implication (ii) $\Rightarrow$ (iii), using the machinery we developed for real operator spaces in the last section, we can replicate the elegant proof due to Paulsen mentioned in \cite[Theorem 4.5.2]{BLM}. Note that the map $\Phi^{'}$ in 
\cite[Theorem 4.5.2]{BLM}, is selfadjoint, therefore by the real version of the Arveson's extension theorem from 
\cite{RuaR}, $\Phi^{'}$ extends to a completely positive and selfadjoint map $\Phi$, on the $C^*$-algebra $M$. By the real version of the Stinespring's Theorem \cite[Theorem 4.3]{RuaR}, the argument in \cite[Proposition 1.3.11]{BLM} can be reproduced, and hence \cite[Proposition 1.3.11]{BLM} holds for real $C^*$-algebras.
Since $\Phi$
fixes the $C^*$-subalgebra 
$$ B= \left [ \begin{array}{ccc} \mathbb{C} & 0 & 0 \\
                             0 & I_{11} & I(X) \\
                             0  & I(X)^{\star} & I_{22} \end{array} \right ]$$
of $M$, so $\Phi$ is a $*$-homomorphism on $B$. By \cite[Proposition 1.3.11]{BLM}, $\Phi: M\longrightarrow M$
 is a bimodule map over $B$. The rest of the argument follows verbatim.

\begin{theorem}\label{multialg}
Let $X$ be a real operator space then $\Ml (X)$ is 
a real operator algebra. Further, $\Al(X)$ is a real $C^*$-algebra.   
\end{theorem} 
\begin{proof}
We use the completely isometric embeddings $X\subset I(X)\subset \mathcal{S}(X)$, and the notation from Section \ref{rie}. Let $$IM_l(X)=\{a\in pI(\mathcal{S}(X))p\ :\ aX\subset X\}.$$
Then  $IM_l(X)$ is a subalgebra of the real $C^*$-algebra $pI(\mathcal{S}(X))p$, and hence is a real operator algebra. Define $\theta:IM_l(X)\longrightarrow \Ml(X)$ as $\theta(a)(x)=ax$ for any $x\in X$. Then $\theta$ is an isometric
isomorphism. Using the canonical identification $M_n(\Ml(X))\cong \Ml(C_n(X))$, define a matrix norm on $M_n(\Ml(X))$ for 
each $n$.   With these matrix norms, and a matricial generalization of the argument after 
Theorem \ref{CharMulti}  (see e.g. \cite[4.5.4]{BLM}),  $\theta$ is a complete isometric isomorphism. Hence all the `multiplier matrix norms' are norms, and
$\Ml(X)\cong IM_l(X)$ is a real operator algebra. Since $\Al(X)=\Ml(X)\cap \Ml(X)^{\star}$, we have that 
$$\Al(X)\cong \{a\in pI(\mathcal{S}(X))p\ :\ aX\subset X {\rm and}\  a^*X\subset X\}.$$ Hence $\Al(X)$ is a real $C^*$-algebra.   
\end{proof}

\begin{theorem}\label{char}
 If $P$ is a projection on a real operator space $X$, then the following are equivalent:
\begin{enumerate}[\upshape (i)]
 \item $P$ is a complete left $M$-projection.
 \item $\tau_P^c$ is completely contractive.
 \item $P$ is an orthogonal projection in the real $C^*$-algebra $\Al(X)$.
\item $P\in \Ml(X)$ with the multiplier norm $\leq 1$.
\item The maps $\nu_P^c$ and $\mu_P^c$ are completely contractive.
\end{enumerate}
\end{theorem}
The above theorem can be easily seen from Proposition \ref{charac1}, Proposition \ref{charac2}, and Theorem \ref{CharMulti}.

\begin{definition}
 A subspace $J$ of a real operator space $X$ is a {\em right $M$-ideal} if $J^{\perp\perp}$ is the range of a 
complete left $M$-projection on $X^{**}$.
\end{definition}

\begin{proposition}
 A projection $P:X\longrightarrow X$ is a complete left $M$-projection if and only if 
$P_c$ is a (complex) complete left $M$-projection on $X_c$.
\end{proposition}
\begin{proof}
We first note that $C_2(X_c)\cong C_2(X)_c$, completely isometrically, via the shuffling map 
$$\left [ \begin{array}{c}
\left[ \begin{array}{cc}  x_1 & -x_2 \\
                           x_2 & x_1 \end{array} \right] \\
\left [ \begin{array}{cc}  y_1 & -y_2\\
                           y_2 & y_1 \end{array} \right ]              
          \end{array}\right ]
       {\mapsto} \left [ \begin{array}{cr}
\left[ \begin{array}{c} x_1 
                        \\ y_1 \end{array} \right ]  & -\left[ \begin{array}{c} x_2 
                                                                               \\ y_2 \end{array} \right ]  \\
  \left[ \begin{array}{c} x_2 
                        \\ y_2 \end{array} \right ]  &  \left[ \begin{array}{c} x_1 
                                                                               \\ y_1 \end{array} \right ]  
                               \end{array}\right ].$$
Also, 
\begin{eqnarray*}
(\tau_P)_c\left( \left [ \begin{array}{c} x\\ v \end{array}\right ]  + i \left [ \begin{array}{c} y\\ w \end{array}\right ] \right )
&=&\left [ \begin{array}{c} P(x)\\ v \end{array}\right ] + i \left [ \begin{array}{c} P(y)\\ w \end{array}\right ] \\
&=&\left [ \begin{array}{c} P(x)+iP(y)\\ v+iw \end{array}\right ] \\
&=&\tau_{(P_c)}\left ( \left [ \begin{array}{c} x+iy\\ v+iw \end{array}\right ]  \right ).
\end{eqnarray*} 
If $P$ is a complete left $M$-projection, then by Theorem \ref{char}, $\tau_P$ and hence, $(\tau_P)_c$ is completely contractive. By the above $\tau_{(P_c)}$ is completely contractive and so, $P_c$ is a complete left $M$-projection. Conversely, if $P_c$ is a complete left $M$-projection, then $\tau_{(P_c)}$ is completely contractive. Since 
$\tau_{(P_c)}|_{C_2(X)}=\tau_P$, $\tau_P$ is a complete contraction and hence $P$ is a complete left $M$-projection, by Theorem \ref{char}.           
\end{proof}

\begin{corollary}\label{corMideal}
 A subspace $J$ in a real operator space $X$ is a right $M$-ideal if and only if $J_c$ is a (complex) right $M$-ideal in $X_c$.
\end{corollary}
\begin{proof}
Since $\left [ \begin{array}{cc} x_t & -y_t \\
                                  y_t & x_t  \end{array} \right ]$ converge ${\rm weak}^*$ in $(X_c)^{**}$ if and only if
both $(x_t)$ and $(y_t)$ converge ${\rm weak}^*$ in $X^{**}$, if $J\subset X$, 
then $(J_c)^{\perp\perp}=\overline{J_c}^{w^*}=(\overline{J}^{w^*})_c=(J^{\perp\perp})_c$. If $J$ is a real right $M$-ideal and if $P: X^{**}\longrightarrow J^{\perp\perp}$ is a (real) left $M$-projection, then by the above corollary $P_c: (X^{**})_c\longrightarrow (J^{\perp\perp})_c$ is a (complex) left $M$-projection. Let $Q$ be the induced map from 
$(X_c)^{**}$ onto $(J_c)^{\perp\perp}$. So the diagram

$$\xymatrix {(X^{**})_c \ar[r]^{P_c} \ar@{<->}[d] & (J^{\perp\perp})_c \ar@{<->}[d] \\
  (X_c)^{**} \ar[r]_Q & (J_c)^{\perp\perp} } $$
commutes and thus $Q$ is an idempotent. Also, since the diagram
 $$\xymatrix {C_2((X^{**})_c) \ar[r]^{\tau_{(P_c)}} \ar@{<->}[d]_{\rm c.i.} & C_2((X^{**})_c) \ar@{<->}[d]^{\rm c.i.} \\
  C_2((X_c)^{**}) \ar[r]_{\tau_Q} &  C_2((X_c)^{**})} $$
  commutes, and $\tau_{(P_c)}$ is a complete contraction, so $\tau_Q$ is a complete contraction. Hence $J_c$ is a right 
$M$-ideal in $X_c$. Conversely, if $P$ is a complete left $M$-projection from $(X_c)^{**}=(X^{**})_c$ onto $(J_c)^{\perp\perp}=(J^{\perp\perp})_c$, 
then let $Q={\rm Re}(P)$. Then a similar argument as in Lemma \ref{lemmainj} shows that $Q$ is an idempotent from 
$X^{**}$ onto $J^{\perp\perp}$. Also since $\tau_Q$ is the restriction of $\tau_{P}$ to $C_2(X^{**})$, $\tau_Q$ is completely contractive. Thus $J$ is a real right $M$-ideal.
\end{proof}

\begin{corollary}
 The right $M$-ideals in a real $C^*$-algebra $A$ are precisely the closed right ideals in $A$.
\end{corollary}

\begin{corollary}
The right $M$-ideals in an approximately unital real operator algebra are precisely the closed right ideals with a left contractive approximate identity.
\end{corollary}

Note that
by  Corollary \ref{corMideal}, and  Proposition \ref{dualcomp}, it is clear that
$X$ is right $M$-ideal in $X^{**}$ if and only if $X_c$ is right $M$-ideal in $(X_c)^{**}$.

We say that a real operator space $X$ is {\bf right $M$-embedded} if $X$ is a right $M$-ideal in $X^{**}$. 
Thus by the above lines it is clear that a real operator space $X$ is right $M$-embedded if and only if $X_c$ is right $M$-
embedded. 
 
\begin{lemma} Let $X$ be a real operator space and $Y\subset X$. Then $(X/Y)_c\cong X_c/Y_c$, completely isometrically.
\end{lemma}
\begin{proof}
Let $\phi:X\longrightarrow X/Y$ be the canonical complete quotient map. We claim that
$\phi: X_c\longrightarrow (X/Y)_c$ given by $x_1+ix_2 \mapsto (x_1+Y)+i (x_2+Y)$ is a complete quotient. Since $\phi$ is a complete contraction, so is $\phi_c$, by \cite[Theorem 2.1]{RuaC}. Thus if we denote the open unit ball of an operator space $Z$ by $\mathcal{U}_Z$, then $\phi_c(\mathcal{U}_{X_c})\subset \mathcal{U}_{(X/Y)_c}$. Let $(x_1+Y)+i (x_2+Y)\in \mathcal{U}_{(X/Y)_c}$, then $\mnorm{\begin{array}{cc} x_1+Y  & x_2 +Y \\ -x_2 +Y  & x_1 +Y\end{array}}< 1$. Since $\phi$ is a complete quotient, $\phi_2: M_2(X)\longrightarrow M_2(X/Y)$ is a quotient and $\phi: \left[\begin{array}{cc} x_1 &  x_2\\ -x_2  & x_1 \end{array}\right]\mapsto \left[\begin{array}{cc} x_1 +Y &  x_2+Y\\ -x_2+Y  & x_1+Y \end{array}\right]$. Thus $\norm{x_1+ix_2}< 1$. Similar argument at each matrix level proves the claim. Since \text{Ker}$(\phi_c)=Y_c$, $(X/Y)_c\cong X_c/Y_c$ completely isometrically.   
\end{proof}

\begin{theorem} Let $X$ be a real right $M$-embedded operator space and $X\subset Y$, then $Y$ and $X/Y$ are right $M$-embedded. 
\end{theorem}
\begin{proof}
Since $Y_c\subset X_c$, by \cite[Theorem 2.6]{Sha2}, $Y_c$ is right $M$-embedded and hence $Y$ is right $M$-embedded. 
Again by \cite[Theorem 2.6]{Sha2}, $X_c/Y_c$ is right $M$-embedded. By the above lemma, $(X/Y)_c$ is right $M$-embedded, and thus, $X/Y$ is right $M$-embedded. 
\end{proof}

\begin{lemma} Let $X_k$ be real operator spaces, then $(\oplus_k^{\circ}X_k)_c\cong \oplus_k^{\circ}(X_k)_c$ completely isometrically.
\end{lemma}
\begin{proof}
Let $\phi: (\oplus_k^{\circ}X_k)_c\longrightarrow \oplus_k^{\circ}(X_k)_c$ be the canonical map $(x_k)+i(y_k)\mapsto (x_k+iy_k)$. Since $\norm{(x_k)}, \norm{(y_k)}\leq \norm{(x_k+iy_k)}\leq \norm{(x_k)}+\norm{(y_k)}$, $\phi$ is well defined and onto. Also, 
\begin{eqnarray*}
\norm{(x_k)+i(y_k)}&=& \mnorm{\begin{array}{cc}(x_k) & (y_k)\\ -(y_k)& (x_k) \end{array}}_{M_2(\oplus_k^{\infty}X_k)}\\
&=& \norm {\left ( \left [ {\begin{array}{cc} x_k & y_k \\ -y_k & x_k \end{array}}\right ]\right )}_{\oplus_k^{\infty}M_2(X_k)}=\norm{(x_k+iy_k)}_{\oplus_k^{\infty}(X_k)_c}.
\end{eqnarray*}  
Thus, $\phi$ is an isometry. Similar argument at each matrix level shows that $\phi$ is a complete isometry.
\end{proof}

The following result follows from above the lemma, \cite[Proposition 2.12]{Sha2}, and the fact that a real operator space $X$ is right $M$-embedded if and only if $X_c$ is right $M$-embedded.

\begin{theorem} Let $A$ be a real right $M$-embedded TRO, then $A\cong \oplus_{i,j} ^{\circ}\mathbb{K}(H_i, H_j)$ completely isometrically,  for some real Hilbert spaces $H_i$, $H_j$.
\end{theorem}

\subsection*{Acknowledgments}

We thank Dr.\ David Blecher, for proposing this project and continually supporting the work. We are grateful for his insightful comments and many suggestions and corrections.

\end{document}